\documentclass{article}

\usepackage{amsmath}
\usepackage{amssymb}
\usepackage{amsthm}
\usepackage{enumerate}
\usepackage{amscd}
\usepackage{url}
\usepackage[T1]{fontenc}

\theoremstyle{plain}

\newtheorem{theorem}{Theorem}
\newtheorem{lemma}[theorem]{Lemma}
\newtheorem{proposition}[theorem]{Proposition}
\newtheorem{corollary}[theorem]{Corollary}

\newtheorem*{corcont}{Corollary~\ref{liste} --- continued}

\theoremstyle{definition}
\newtheorem{remark}[theorem]{Remark}

\usepackage{url}

\newcommand{\beq}{\begin{equation}}
\newcommand{\eeq}{\end{equation}}

\DeclareMathOperator{\Sym}{Sym}

\DeclareMathOperator{\tr}{tr}

\newcommand{\PP}{{\mathbb{P}}}
\newcommand{\Z}{\mathbb{Z}}
\newcommand{\N}{\mathbb{N}}
\newcommand{\Q}{\mathbb{Q}}

\newcommand{\F}{\mathbb{F}}

\newcommand{\longto}{\longrightarrow}
\newcommand{\tens}{\otimes}
\newcommand{\cA}{\mathcal{A}}

\begin{document}

\title{Gaps between prime numbers and tensor rank of multiplication in finite fields}


\author{Hugues Randriam\footnote{supported by
ANR-14-CE25-0015 project Gardio and ANR-15-CE39-0013 project Manta}}
\date{}

\maketitle

\begin{abstract}
We present effective upper bounds
on the symmetric bilinear complexity of multiplication in extensions of
a base finite field $\F_{p^2}$ of prime square order,
obtained by combining estimates on gaps between prime numbers
together with an optimal construction of auxiliary divisors for
multiplication algorithms by evaluation-interpolation on curves.
Most of this material dates back to a 2011 unpublished work of the author,
but it still provides the best results on this topic at the present time.

Then a few updates are given in order to take recent developments into account,
including comparison with a similar work of Ballet and Zykin,
generalization to classical bilinear complexity over $\F_p$,
and to short multiplication of polynomials,
as well as a discussion of open questions on gaps between prime numbers or more
generally values of certain arithmetic functions.

\vspace{.5\baselineskip}


\end{abstract}

\section{Introduction}

Let $F$ be a field and $\cA$ a finite dimensional commutative $F$-algebra.
Denote by $m_{\cA}:\cA\times\cA\longto\cA$ the multiplication map in $\cA$,
seen as a symmetric $F$-bilinear map,
and by $T_{\cA}\in\Sym^2(\cA^\vee)\tens_F\cA$ the associated tensor,
where $\cA^\vee$ is the dual space of $\cA$ over $F$
and $\Sym^2(\cA^\vee)\subset\cA^\vee\tens_F\cA^\vee$ stands for the subspace of symmetric tensors.

By a symmetric bilinear multiplication algorithm for $\cA$, of length $n$,
we mean one of the following equivalent data
(see e.g. \cite{HR-JComp} or \cite[\S5.1--5.3]{HR-AGCT}):
\begin{itemize}
\item linear maps $\alpha:\cA\longto F^n$ and $\omega:F^n\longto\cA$
such that the following diagram commutes
\nopagebreak
\beq
\begin{CD}
\cA\times\cA @>{m_{\cA}}>> \cA \\
@V{(\alpha,\alpha)}VV @AA{\omega}A \\
F^n\times F^n @>{*}>> F^n
\end{CD}
\eeq
where $*$ denotes componentwise multiplication in $F^n$
\item linear forms $\alpha_1,\dots,\alpha_n:\cA\longto F$
and elements $w_1,\dots,w_n\in\cA$ such that
the product in $\cA$ of any two $x,y\in\cA$ can be computed as
\beq
xy=\sum_{1\leq i\leq n}\alpha_i(x)\alpha_i(y)w_i
\eeq
\item a decomposition
\beq
T_\cA=\sum_{1\leq i\leq n}\alpha_i^{\tens 2}\tens w_i
\eeq
of $T_\cA$ as a sum of $n$ elementary symmetric tensors in $\Sym^2(\cA^\vee)\tens\cA$.
\end{itemize}

We define the \emph{symmetric bilinear complexity} of $\cA$ over $F$
\beq
\mu^{\operatorname{sym}}_F(\cA)
\eeq
as the smallest length $n$ for which such
a symmetric bilinear multiplication algorithm exists.
Equivalently, $\mu^{\operatorname{sym}}_F(\cA)$ is the \emph{symmetric tensor rank} of $T_\cA$.

Most of our interest will be when $F=\F_q$ is a finite field
and $\cA=\F_{q^k}$ is its (unique) degree $k$ field extension.
We then set
\beq
\mu^{\operatorname{sym}}_q(k)=\mu^{\operatorname{sym}}_{\F_q}(\F_{q^k}).
\eeq

We will focus on upper bounds for this quantity or, which is essentially
the same, on the construction of symmetric bilinear multiplication algorithms
for $\F_{q^k}$ over $\F_q$, especially when $k$ is large.
For this, a powerful method was introduced by Chudnovsky and Chudnovsky
in 1987 with \cite{ChCh87}\cite{ChCh88}, using evaluation-interpolation on
algebraic curves.

When looking at the literature, the reader should be wary that these authors,
and those who followed them, actually expressed their results in terms
of the classical (possibly asymmetric) bilinear complexity $\mu_q(k)$,
not in terms of the symmetric bilinear complexity $\mu^{\operatorname{sym}}_q(k)$.
Indeed, this last notion was first introduced in this context
only in 2012 with \cite{HR-JComp}.
However, the original construction of \cite{ChCh87}\cite{ChCh88} naturally
produces symmetric algorithms.
This allows us, in what follows, to restate the bounds derived by this method
in terms of $\mu^{\operatorname{sym}}_q(k)$, even if the original statements were
in terms of $\mu_q(k)$.

The first and probably the most spectacular achievement of this method
is the proof that this quantity asymptotically grows linearly with $k$. 
Indeed, when recast in terms of symmetric complexity, the main result
of \cite{ChCh87}\cite{ChCh88} reads as:
\beq
\label{asympChCh}
\limsup_{k\to\infty}\frac{1}{k}\mu^{\operatorname{sym}}_q(k)\leq 2\left(1+\frac{1}{\sqrt{q}-3}\right)\qquad\textrm{for $q\geq25$ a square.}
\eeq
Actually, parts of the proof given for this result were somehow sketchy,
but all the missing details were later provided by Shparlinski, Tsfasman, and
Vladut in 1991 with \cite{STV}.

At the same time they provided these missing details, these same authors also
proposed the following improved bound:
\beq
\label{fSTV}
\limsup_{k\to\infty}\frac{1}{k}\mu^{\operatorname{sym}}_q(k)\leq 2\left(1+\frac{1}{\sqrt{q}-2}\right)\qquad\textrm{for $q\geq9$ a square.}
\eeq
Unfortunately, a fatal flaw was later found in their proof,
as first observed in \cite{Cascudo}.
This error concerns the solution of what some authors now call ``Riemann-Roch
systems of equations'' \cite{CCX}, a key ingredient in the Chudnovsky-Chudnovsky
method, and it totally invalidates the proof given for \eqref{fSTV}.

Fortunately, an alternative method that allows to solve such Riemann-Roch systems was
then discovered by the author around end of 2010, and published in 2013
with \cite{21sep}.
It readily allows to repair the proof of \eqref{fSTV},
with only one small downfall: the method only applies to sligthly
larger $q$ than originally needed.

It thus became desirable to try to fine tune the method of \cite{21sep}
in order to make it work for $q$ as small as possible.
This was the main goal of \cite{2D-G},
and it allowed to repair the Shparlinski-Tsfasman-Vladut bound \eqref{fSTV}
as follows.

For any prime power $q$, define the \emph{dense Ihara constant} \cite[p.~23]{2D-G}
as the smallest real number $A'(q)$ such that there exists a sequence
of curves $X_j$ over $\F_q$, of genus $g_j\to\infty$,
with
\begin{itemize}
\item $\frac{|X_j(\F_q)|}{g_j}\to A'(q)$
\item $\frac{g_{j+1}}{g_j}\to 1$
\end{itemize}
as $j\to\infty$.
Then \cite[Cor.~18]{2D-G} we have
\beq
\label{STVA'}
\limsup_{k\to\infty}\frac{1}{k}\mu^{\operatorname{sym}}_q(k)\leq 2\left(1+\frac{1}{A'(q)-1}\right)\qquad\textrm{as soon as $\textstyle{A'(q)\geq 5-\frac{14q^2-4}{q^4+2q^2-1}}$,}
\eeq
and in particular we have
\beq
\label{rSTV}
\limsup_{k\to\infty}\frac{1}{k}\mu^{\operatorname{sym}}_q(k)\leq 2\left(1+\frac{1}{\sqrt{q}-2}\right)\qquad\textrm{for $q\geq49$ a square.}
\eeq
Actually, \eqref{rSTV} does not requires the full strength of \cite{2D-G}. 
It can readily be derived from the original, simpler results of \cite{21sep}.
Thus, while \cite{2D-G} remained unpublished, \eqref{rSTV}
can be found in published form, with full proof, in \cite[Th.~6.4]{HR-JComp}.
Observe that it entirely repairs \eqref{fSTV}, but for $q\geq49$ instead of $q\geq9$,
leaving only the cases $q=9,16,25$ uncovered.

\vspace{\baselineskip}

This far we discussed only works on \emph{asymptotic} upper bounds.
Parallel to these, some authors considered \emph{effective} upper
bounds on $\mu^{\operatorname{sym}}_q(k)$,
that should apply to any (finite, explicit) value of $k$.
This topic was treated by Ballet first in \cite{Ballet1999}
and \cite{Ballet2003}, and then improved in \cite{Ballet2008}.
However it turned out this last work contained several errors,
the most important of which being that it reproduced the
flawed proof from \cite{STV} and based all its results on it.

Thus, in \cite{2D-G}, while repairing the Shparlinski-Tsfasman-Vladut bound,
the author also explained how his method could repair Ballet's
results likewise. 
Actually, not only these results could be repaired, but they could also
be improved.
Indeed, part of Ballet's argument was based on Bertrand's postulate,
proved by Chebyshev, that asserts that for any real $x>1$, there is
a prime between $x$ and $2x$.
It was clear that improved bounds on the bilinear complexity
could be derived from finer estimates on the gaps between prime numbers
(such as \cite{BHP}).

This was presented in section~5 of \cite{2D-G}, somehow as a digression
(it is also discussed in \cite{HR-JComp}, especially Rem.~5.3, 5.5, 5.8,
but with a slightly different approach).
Unfortunately, the fact that \cite{2D-G} remained unpublished,
in French, and the
fact that this section~5 followed long and technical developments in a
quite unrelated direction, did not help disseminate the ideas
introduced there.

Ultimately, this method, which combines
\begin{enumerate}[(a)]
\item the author's optimal solution of Riemann-Roch systems for the Chudnovsky-Chudnovsky method
\item fine estimates on the gaps between prime numbers,
\end{enumerate}
and which still provides the best effective upper bounds
on $\mu^{\operatorname{sym}}_q(k)$ when $q=p^2$ is a prime square,
seemed to have been forgotten by the experts.
As an illustration, very recently Ballet and Zykin \cite{BZ} partially
rediscovered this method (ingredient (b) only, not (a)), but a preliminary
version of their work did not even mention \cite{2D-G}
--- fortunately this is corrected in the final version.

Thus, almost seven years after it was first written,
the author would like to take this opportunity to finally
publish these results in the peer-reviewed literature.
Hopefully this will provide a proper reference for future researchers.
Accordingly, the next section is a translation into English of section~5
of \cite{2D-G}, with essentially no significant change.
Then in the last section we present some updates in order to take recent
developments into account, and in particular we explain the links
with \cite{BZ}. We also discuss a few questions presented in \cite{2D-G}.
This includes compatibilty with generalizations of the Chudnovsky-Chudnovsky method \cite{HR-JComp},
leading to new effective and asymptotic bounds
on the classical bilinear complexity $\mu_q(k)$
when $q=p$ is prime,
and to similar results for short multiplication of polynomials;
as well as open questions on gaps between prime numbers or more
generally gaps in the set of values of certain arithmetic functions.

\section{Main results as of 2011}

As explained in the introduction, our aim here is to fix the proof of the main result
claimed by Ballet in \cite{Ballet2008}, and then to improve on it.
This statement concerns effective upper bounds on $\mu^{\operatorname{sym}}_{p^2}(k)$,
the symmetric bilinear multiplication complexity in extensions of a base
field $\F_{p^2}$ of prime square order.

We first recall an instance of the basic construction of Chudnovsky and Chudnovsky \cite{ChCh87}\cite{ChCh88}:
\begin{proposition}
\label{basicChCh}
Let $X$ be a curve of genus $g$ over the finite field $\F_q$,
equipped with a closed point $Q$ of degree~$k$,
and with $n$ points $P_1,\dots,P_n$ of degree~$1$.
Suppose that $X$ also admits a $\F_q$-rational divisor $D$
(w.l.o.g. of support disjoint from $Q$ and the $P_i$)
such that 
\begin{itemize}
\item $D-Q$ is nonspecial\\ (so the evaluation map $L(D)\longto \F_q(Q)=\F_{q^k}$ is surjective)
\item $2D-(P_1+\cdots+P_n)$ has no section\\ (so the evaluation map $L(2D)\longto\bigoplus_{i=1}^n\F_q(P_i)=\F_q^n$ is injective). 
\end{itemize}
Then there is a symmetric bilinear multiplication algorithm of length $n$ for $\F_{q^k}$ over $\F_q$, i.e.
\beq
\mu^{\operatorname{sym}}_q(k)\leq n.
\eeq
\end{proposition}

Observe that $D-Q$ nonspecial implies $\deg(D)-k\geq g-1$,
and $2D-(P_1+\cdots+P_n)$ without section implies $2\deg(D)-n\leq g-1$.
So combining both, we see a necessary condition for the existence of such $D$, $Q$, and $P_i$
is that $X$ admits at least
\beq
\label{optimal}
|X(\F_q)|\geq n\geq2k+g-1
\eeq
points of degree~$1$.

We say a method for finding such data on $X$ is \emph{optimal} if it can work
with equality attained in \eqref{optimal}.

In the course of the proof of their main result in \cite{ChCh87}\cite{ChCh88},
Chudnovsky and Chudnovsky constructed such $D$, $Q$, and $P_i$,
but only under the suboptimal condition
\beq
|X(\F_q)|\geq2k+2g-1.
\eeq
This was also stated more explicitely by Ballet as \cite[Lemma~2.2]{Ballet1999}.
Roughly speaking, the construction proceeds by first fixing $Q$ and $D$, and then finding the $P_i$.

By contrast, in order to reach optimality, it is more natural to first fix $Q$
and $G=P_1+\dots+P_n$, and then only look for $D$
such that $D-Q$ is nonspecial and $2D-G$ has no section.
Seen this way, the problem essentially reduces to a ``Riemann-Roch system of equations''
in the sense of \cite{CCX}.

In \cite[pp.~161--162]{STV} a solution to this Riemann-Roch system is proposed under the optimal
condition \eqref{optimal}. Unfortunately, an error in the proof was detected by Cascudo
in \cite{Cascudo}, which invalidates the argument.
It turns out the very same result was later stated also
by Ballet as \cite[Prop.~2.1]{Ballet2008},
with the same proof and the error it contains reproduced as well.
Let us briefly explain this error: assuming $n=2k+g-1$,
the core of the argument is to show that the number of divisor classes $[D]$
such that $2D-G$ has sections is not more than the number of effective divisors of degree $g-1$;
for this, to each such class, one assigns an effective divisor $E\sim 2D-G$,
and one concludes with the claim that this map $[D]\mapsto E$ is injective;
unfortunately this last claim is false in general:
indeed, if the class group has some $2$-torsion,
it could happen that two divisors $D$ and $D'$ are not linearly equivalent,
but $2D-G$ and $2D'-G$ are, and give the same $E$.

\vspace{\baselineskip}

Fortunately, in \cite{21sep} the author introduced a new construction that provides an optimal
solution to certain Riemann-Roch systems.
In \cite{2D-G} it was applied to the system associated with the Chudnovsky-Chudnovsky method,
which allows to substitute Ballet's erroneous \cite[Prop.~2.1]{Ballet2008}
with the following \cite[Cor.~20]{2D-G}:
\begin{proposition}
\label{constrD-Qet2D-G}
Let $X$ be a curve of genus $g$ over
a finite field $\F_q$, equipped with two $\F_q$-rational
divisors $Q$ and $G$. Set $k=\deg Q$ and $n=\deg G$, and assume
\beq
\label{5g}
|X(\F_q)|>5g
\eeq
and
\beq
n\geq2k+g-1.
\eeq
Then there exists a $\F_q$-rational divisor $D$ on $X$,
with support in $X(\F_q)$, such that 
$D-Q$ is nonspecial of degree $g-1$,
and $2D-G$ has no section.

In particular, if $n=2k+g-1$, then
both $D-Q$ and $2D-G$ are nonspecial of degree $g-1$.
\end{proposition}
The only downfall is the new condition \eqref{5g}, but this does not cause any trouble
unless $q$ is very small.
In particular it does not hinder optimality,
so it will be sufficient for us in order to fix Ballet's result.
Actually, this condition \eqref{5g} could be slightly relaxed,
using the machinery introduced in sections~1--2 of \cite{2D-G}
precisely for this.
But stated this way, Proposition~\ref{constrD-Qet2D-G} is a simplified version
that does not use the full strength of \cite{2D-G}, and could be derived
directly from the original results of \cite{21sep}.

Then, Ballet's \cite[Th.~2.1(1)]{Ballet2008} is replaced with 
the following \cite[Lemma~21]{2D-G}:
\begin{lemma}
\label{mu2k+g-1}
Let $X$ be a curve of genus $g$ over a finite field $\F_q$
with
\beq
|X(\F_q)|>5g.
\eeq
Then for all integers $k$ in the interval
\beq
\label{encadrek}
\left\lceil2\log_q\frac{2g+1}{\sqrt{q}-1}\right\rceil\:<\: k\;\leq\:\frac{|X(\F_q)|+1-g}{2}
\eeq
we have
\beq
\mu^{\operatorname{sym}}_q(k)\leq 2k+g-1.
\eeq
\end{lemma}
\begin{proof}
Following \cite{2D-G}, this is a direct consequence of Proposition~\ref{basicChCh},
together with \cite[Cor.~V.2.10.c]{Stichtenoth} and Proposition~\ref{constrD-Qet2D-G} above.

Alternatively, in order to refer to a published source, we observe it is
also a special case
of \cite[Th.~5.2(c)]{HR-JComp} applied with $m=k$, $l=1$, $n_{1,1}=2k+g-1$,
and $n_{d,u}=0$ for other values of $d,u$.
\end{proof}

For instance, taking $X=\PP^1$, we find:
\beq
\mu^{\operatorname{sym}}_q(k)\leq2k-1\qquad\textrm{for $k\leq\frac{q}{2}+1$,}
\eeq
an inequality that is in fact easily seen to be an equality \cite{Winograd}.

Likewise, choosing for $X$ a suitable elliptic curve, yields the following bound from \cite{Shokro}:
\beq
\label{inegShokro}
\mu^{\operatorname{sym}}_q(k)\leq 2k\qquad\textrm{for $k<\frac{q+e(q)+1}{2}$}
\eeq
with $e(q)\lesssim 2\sqrt{q}$, and in particular $e(q)=2\sqrt{q}$
if $q$ is a square.

One could continue in the same way with curves of genus $2$, $3$, etc.

Another equivalent point of view is the following.
For any integer $k$, let
$\mathcal{X}_{q,k}$ be the set of curves (up to isomorphism)
$X$ over $\F_q$, of genus $g=g(X)$, satisfying:
\begin{enumerate}[(a)]
\item $g\leq\frac{1}{2}(q^{(k-1)/2}(q^{1/2}-1)-1)$
\item $|X(\F_q)|>5g$
\item $|X(\F_q)|\geq 2k+g-1$.
\end{enumerate}
Then:
\begin{lemma}
\label{muXqk}
For any finite field $\F_q$,
and for any integer $k$ such that $\mathcal{X}_{q,k}$ is nonempty,
we have
\beq
\frac{1}{k}\mu^{\operatorname{sym}}_q(k)\leq 2+\frac{\min_{X\in\mathcal{X}_{q,k}}g(X)-1}{k}.
\eeq
\end{lemma}
\begin{proof}
It is a reformulation of the previous lemma.
\end{proof}

Compared to similar results in the literature,
our equivalent Lemma~\ref{mu2k+g-1} and Lemma~\ref{muXqk} impose less restriction
between $k$, $g$, and the number of points on the curve.
For instance, \cite[Th.~1.1 and Cor.~2.1]{Ballet1999} reach the same conclusion,
but only for $k\leq\frac{|X(\F_q)|+1-2g}{2}$ instead of the second inequality in \eqref{encadrek},
or equivalently, under the stronger condition $|X(\F_q)|\geq 2k+2g-1$ instead of (c)
in the definition of $\mathcal{X}_{q,k}$.
However, our method requires curves with ``sufficiently many'' points, as expressed by condition (b).


\vspace{\baselineskip}

Now we can go on with the same arguments as in \cite{Ballet2008},
and then improve on the result that is stated there.

Consider the Dedekind psi function, defined for any integer $N$ by
\beq
\psi(N)=N\prod_{\substack{l|N\\ \textrm{$l$ prime}}}\left(1+\frac{1}{l}\right).
\eeq

\begin{lemma}
\label{courbesmodulaires}
Let $p$ be a prime number, and $N$ an integer prime to $p$.
Then the modular curve
$X_0(N)$ is smooth over $\F_p$, of genus
\beq
\label{g0N<psi}
g_0(N)\leq\frac{\psi(N)}{12},
\eeq
and it admits
\beq
\label{|X0N|>psi}
|X_0(N)(\F_{p^2})|\geq(p-1)\frac{\psi(N)}{12}
\eeq
points over $\F_{p^2}$.
\end{lemma}
\begin{proof}
See \cite{TV}, \S~4.1.
\end{proof}

\begin{remark}
Actually we can be slightly more precise in this lemma.
Hurwitz's formula gives an exact expression
\beq
g_0(N)=\frac{\psi(N)}{12}+1-\frac{\nu_\infty(N)}{2}-\frac{\nu_3(N)}{3}-\frac{\nu_2(N)}{4}
\eeq
where
\begin{itemize}
\item $\;\nu_\infty(N)=\sum_{d|N}\phi(\gcd(d,\frac{N}{d}))=\displaystyle\prod_{l^\nu||N}\textstyle\begin{cases}2l^{\frac{\nu-1}{2}} & \textrm{if }\nu\textrm{ odd}\\ (l+1) l^{\frac{\nu}{2}-1} & \textrm{if }\nu\textrm{ even}\end{cases}$
\item $\;\nu_3(N)=\begin{cases}\prod_{l|N}\left(1+\left(\frac{-3}{l}\right)\right) & \textrm{if }9\nmid N \\ 0 & \textrm{if }9\,|\,N\end{cases}$
\item $\;\nu_2(N)=\begin{cases}\prod_{l|N}\left(1+\left(\frac{-1}{l}\right)\right) & \textrm{if }4\nmid N \\ 0 & \textrm{if }4\,|\,N\end{cases}$
\end{itemize}
while the Eichler-Shimura relation gives
\beq
|X_0(N)(\F_{p^2})|=p^2+1+pg_0(N)-\tr T_{p^2}
\eeq
where the Hecke operator $T_{p^2}$ acts on the space of cusp forms $S_2(\Gamma_0(N))$,
and its trace can be computed explicitly, e.g. by the formula given in \cite{Miyake},
Th.~6.8.4 and Rem.~6.8.1, pp.~263--264:
\beq
\tr T_{p^2}=\frac{\psi(N)}{12}+\delta(N,p^2)-\sum_ta(t)\sum_fb(t,f)c(t,f)
\eeq
where $\delta(N,p^2)=p^2+p+1$ if $N>1$.
The terms $a(t)\sum_fb(t,f)c(t,f)$ are nonnegative, and their contribution
to the sum has a simple expression for certain special values of $t$:
\begin{itemize}
\item $\;\frac{1}{2}\,p\,\nu_\infty(N)$ for $t=\pm 2p$
\item $\;\frac{1}{3}\left(p+1-\left(\frac{-3}{p}\right)\right)\nu_3(N)$ for $t=\pm p\;$ (if $3\nmid N$)
\item $\;\frac{1}{4}\left(p+1-\left(\frac{-1}{p}\right)\right)\nu_2(N)$ for $t=0\;$ (if $2\nmid N$)
\end{itemize}
so that for $N>1$ prime to $6p$:
\beq
|X_0(N)(\F_{p^2})|\geq(p-1)\frac{\psi(N)}{12}\,+\,\frac{1\!-\!\left(\frac{-3}{p}\right)}{3}\nu_3(N)+\frac{1\!-\!\left(\frac{-1}{p}\right)}{4}\nu_2(N)
\eeq
Similar formulas can be derived for $2|N$ or $3|N$.
\end{remark}

For any infinite subset $\mathcal{A}$ of $\N$ and for any real $x>0$,
let
\beq
\lceil x\rceil_{\mathcal{A}}=\min\:\mathcal{A}\cap[x,+\infty[
\eeq
be the smallest element of $\mathcal{A}$ larger than or equal to $x$.
Also set
\beq
\epsilon_{\mathcal{A}}(x)=\sup_{y\geq x}\:\frac{\lceil y\rceil_{\mathcal{A}}-y}{y},
\eeq
so the function $\epsilon_{\mathcal{A}}$ is monotonously non-increasing,
and for any $x>0$,
the interval $[x,(1+\epsilon_{\mathcal{A}}(x))x]$ contains an element of $\mathcal{A}$.


For instance, if $p$ is a prime number, then
$\lceil x\rceil_{\psi(\N\setminus p\N)}$ is the smallest integer $n\geq x$
that can be written as $n=\psi(N)$ for an integer $N$ prime to $p$, and:

\begin{lemma}
\label{psiBertrand}
With these notations, for $p\neq2$ we have
\beq
\lceil x\rceil_{\psi(\N\setminus p\N)}\leq 2x\quad\textrm{for all $x\geq\frac{3}{2}$,}
\eeq
or said otherwise:
\beq
\epsilon_{\psi(\N\setminus p\N)}(3/2)\leq 1.
\eeq
\end{lemma}
\smallskip
\begin{proof}
Indeed, for
$j=\lfloor\frac{\log 2x/3}{\log 2}\rfloor$,
we have $x<3\cdot2^j=\psi(2^{j+1})\leq2x$.
\end{proof}

\begin{proposition}
\label{Ballet+}
Let $p\geq 7$ be a prime number.
Then for all $k>\frac{p^2+p+1}{2}$ we have
\beq
\frac{1}{k}\mu^{\operatorname{sym}}_{p^2}(k)\leq 2+\frac{\frac{1}{12}\left\lceil\frac{24k-12}{p-2}\right\rceil_{\psi(\N\setminus p\N)}-1}{k}.
\eeq
\end{proposition}
\begin{proof}
Choose $N$ prime to $p$ such that
$\psi(N)=\left\lceil\frac{24k-12}{p-2}\right\rceil_{\psi(\N\setminus p\N)}$
and set $X=X_0(N)$.
Then, by \eqref{g0N<psi} and \eqref{|X0N|>psi} we have $|X(\F_{p^2})|-g\geq(p-2)\frac{\psi(N)}{12}$,
so condition (c) before Lemma~\ref{muXqk} is satisfied.

Likewise we have $|X(\F_{p^2})|-5g\geq(p-6)\frac{\psi(N)}{12}$,
so for $p\geq7$ condition (b) is satisfied too.


Last, by Lemma~\ref{psiBertrand} we have
$\psi(N)=\left\lceil\frac{24k-12}{p-2}\right\rceil_{\psi(\N\setminus p\N)}\leq\frac{48k-24}{p-2}$
so
\beq
g\leq\frac{\psi(N)}{12}\leq\frac{4k-2}{p-2},
\eeq
and for $p\geq7$ and $k>\frac{p^2+p+1}{2}$, this last quantity is easily
shown to be less than $\frac{1}{2}(p^{k-1}(p-1)-1)$.
Thus condition (a) is satisfied, and we conclude with Lemma~\ref{muXqk}.
\end{proof}

\begin{remark}
\label{remgapspsi}
Thanks to this proposition, any (effective) upper bound
on the function $\lceil .\rceil_{\psi(\N\setminus p\N)}$,
or on $\epsilon_{\psi(\N\setminus p\N)}$,
translates into an (effective) upper bound on the $\mu^{\operatorname{sym}}_{p^2}(k)$.
Our task is then, for any given real $x>0$, to find an integer $N$
prime to $p$ such that $\psi(N)$ is larger than or equal to $x$ but as
small as possible. A quick analysis suggests two natural approaches to this problem.

First, one can look for $N$ among integers having only small prime factors.
Indeed, let $\mathcal{B}=\{l_1,\dots,l_B\}$ be a set of prime numbers, $p\not\in\mathcal{B}$.
Set $N_{\mathcal{B}}=\prod_{i=1}^Bl_i$
and assume $\psi(N_{\mathcal{B}})=\prod_{i=1}^B(l_i+1)<x$.
Then if $N=N'N_{\mathcal{B}}$ where $N'$ has all its prime factors
in $\mathcal{B}$, we have $\psi(N)=N'\psi(N_{\mathcal{B}})$.
Thus, if we can find an integer $N'\geq\frac{x}{\psi(N_{\mathcal{B}})}$
as small as possible
with all its prime factors in $\mathcal{B}$,
we deduce an upper bound on $\lceil x\rceil_{\psi(\N\setminus p\N)}$.
For $\mathcal{B}=\{2\}$ this is precisely Lemma~\ref{psiBertrand}.
It would be interesting to optimize the choice $\mathcal{B}$
(possibly depending on $x$) in order to get better estimates.

At the opposite, one can choose $N$ among integers having
only large prime factors. Indeed, if $N$
has no prime factor smaller than $N^{1/u}$,
then $\psi(N)\leq N\left(1+\frac{1}{N^{1/u}}\right)^u$,
and if we can produce such an $N\geq x$ as small as possible,
then, for a convenient choice of $u$, one could hope to get a
bound close enough to $\lceil x\rceil_{\psi(\N\setminus p\N)}$.
The extreme case is $u=1$, which means we look only at $N$ prime.
We then get the upper bound
\beq
\label{majpremier}
\lceil x\rceil_{\psi(\N\setminus p\N)}\leq\lceil x-1\rceil_{\mathcal{P}}+1\quad\textrm{for $x>p+1$}
\eeq
where $\mathcal{P}$ is the set of prime numbers
(indeed, $N=\lceil x-1\rceil_{\mathcal{P}}$ is a prime number larger than $p$,
and $\psi(N)=N+1\geq x$).
This allows to use all known results on the function $\epsilon_{\mathcal{P}}$;
for instance, Bertrand's postulate, proved by Chebyshev,
gives $\epsilon_{\mathcal{P}}(1)=1$, and combined with \eqref{majpremier},
it provides essentially the same bound as in Lemma~\ref{psiBertrand}. 
Several sharper bounds on $\epsilon_{\mathcal{P}}$ are known, and we list
their consequences in the corollary below.
However, here again, it would still be interesting to study
whether a convenient choice of $u>2$ (possibly depending on $x$),
would give significantly better.
\end{remark}

\begin{corollary}
\label{liste}
Let $p\geq 7$ be a prime number.
Then
\begin{enumerate}[(i)]
\item for all $k>\frac{p^2+p+1}{2}$,
\beq
\frac{1}{k}\mu^{\operatorname{sym}}_{p^2}(k)\leq 2\left(1+\frac{1+\epsilon_{\mathcal{P}}\!\left(\frac{24k}{p-2}\right)}{p-2}\right)
\eeq
\item for all $k\geq1$,
\beq
\frac{1}{k}\mu^{\operatorname{sym}}_{p^2}(k)\leq 2\left(1+\frac{2}{p-2}\right)
\eeq
\item for all $k\geq1$,
\beq
\frac{1}{k}\mu^{\operatorname{sym}}_{p^2}(k)\leq 2\left(1+\frac{1+\frac{10}{139}}{p-2}\right)
\eeq
\item for all $k\geq e^{50}p$,
\beq
\frac{1}{k}\mu^{\operatorname{sym}}_{p^2}(k)\leq 2\left(1+\frac{1.000\,000\,005}{p-2}\right)
\eeq
\item for all $k\geq 16\,531\,(p-2)$,
\beq
\frac{1}{k}\mu^{\operatorname{sym}}_{p^2}(k)\leq 2\left(1+\frac{1+\frac{1}{25\log^2\frac{24k}{p-2}}}{p-2}\right)
\eeq
\item for all $k$ large enough,
\beq
\frac{1}{k}\mu^{\operatorname{sym}}_{p^2}(k)\leq 2\left(1+\frac{1+\frac{1}{\left(\frac{24k}{p-2}\right)^{0.475}}}{p-2}\right).
\eeq
\end{enumerate}
\end{corollary}
\begin{proof}
Item (i) follows from Proposition~\ref{Ballet+}, from \eqref{majpremier},
and the obvious inequality
$\left\lceil\frac{24k-12}{p-2}-1\right\rceil_{\mathcal{P}}\leq\left\lceil\frac{24k}{p-2}\right\rceil_{\mathcal{P}}$.

Item (ii) follows from \eqref{inegShokro},
Proposition~\ref{Ballet+}, and Lemma~\ref{psiBertrand}.

Noting that for $p\geq 7$ and
$k>\frac{p^2+p+1}{2}$ we have $\frac{24k}{p-2}>139$,
item (iii) follows from \eqref{inegShokro}, from (i),
and from $\epsilon_{\mathcal{P}}(139)=10/139$.
To justify this last equality, observe that if
$p_1=2$, $p_2=3$, $p_3=5$, $p_4=7$, $p_5=11$, \dots {}
is the sequence of prime numbers,
then for all $n\leq n'$ we have
\beq
\epsilon_{\mathcal{P}}(p_n)=\max\left(\epsilon_{\mathcal{P}}(p_{n'}),\,\max_{n\leq j<n'}\frac{p_{j+1}-p_j}{p_j}\right).
\eeq
Set $p_n=139$, estimate $\epsilon_{\mathcal{P}}(p_{n'})$ for
$p_{n'}=2\,010\,881$ using \cite{Schoenfeld}
(or for $p_{n'}=396\,833$ using \cite{Dusart})
and conclude by explicitly computing the (finitely many) remaining terms
for $n\leq j<n'$.

Likewise, items (iv), (v) and (vi) follow from (i)
and the estimates on $\epsilon_{\mathcal{P}}$ that are given in
\cite{RS},
\cite{Dusart} (preprint version only, beware that the published verion is different),
and \cite{BHP}, respectively. 
\end{proof}

\section{More recent developments, and questions that remain open}

\noindent\textbf{3.1. New estimates on gaps between primes.}
Corollary~\ref{liste}(i) allows to systematically translate any estimate
on gaps between primes into a bound on $\mu^{\operatorname{sym}}_{p^2}(k)$.
In Corollary~\ref{liste}(iii)-(vi) we listed such bounds, based on the
state of the literature in 2011, i.e. at the time when \cite{2D-G} was written.

Certainly many new results of this type have been published since then,
and will be published in the future.
One such result is Dudek's \cite{Dudek}, that has been used by Ballet and Zykin \cite{BZ} (see \S 3.5 below),
and which asserts that for any real $x>e^{e^{33.3}}$ there is a prime between $x$ and $x+3x^{2/3}$,
or with our notations, $\epsilon_{\mathcal{P}}(x)\leq 3x^{-1/3}$.
Combined with Corollary~\ref{liste}(i), this gives at once:
\begin{corcont}(vii) For $p\geq7$ and $k\geq\frac{p-2}{24}e^{e^{33.3}}$,
\beq
\frac{1}{k}\mu^{\operatorname{sym}}_{p^2}(k)\leq 2\left(1+\frac{1+\frac{3}{\left(\frac{24k}{p-2}\right)^{1/3}}}{p-2}\right).
\eeq
\end{corcont}

Actually, this Corollary~\ref{liste}(vii) 
is weaker than Corollary~\ref{liste}(vi)
because the exponent $1/3$ is not as good as the exponent $0.475$.
But it is fully effective, in the sense that the range of $k$ for which
it holds is given explicitely from Dudek's work \cite{Dudek},
while in \cite{BHP} only the existence is proved (although the authors
observe it could be made explicit with enough work).

\vspace{\baselineskip}

How far could we hope to go with this method?
It is known that, under the Riemann hypothesis,
we should have $\epsilon_{\mathcal{P}}(x)=\widetilde{O}(x^{-1/2})$.
Combined with Corollary~\ref{liste}(i), this gives:
\beq
\frac{1}{k}\mu^{\operatorname{sym}}_{p^2}(k)\leq 2\left(1+\frac{1+\widetilde{O}(k^{-1/2})}{p-2}\right).
\eeq
Ultimately, it is conjectured $\epsilon_{\mathcal{P}}(x)=O(\log^2(x)/x)$ \cite{Cramer},
which would give likewise:
\beq
\frac{1}{k}\mu^{\operatorname{sym}}_{p^2}(k)\leq 2\left(1+\frac{1+O(\log^2(k)/k)}{p-2}\right).
\eeq

\vspace{\baselineskip}
\noindent\textbf{3.2. Gaps between prime numbers in a given residue class.}
(This is a translation of the paragraph at the bottom of \cite[p.~31]{2D-G},
and is also discussed in \cite[Rem.~5.5]{HR-JComp}.)

Our main results concerned a base field $\F_q$ where $q=p^2$ 
is a prime square, and used evaluation-interpolation on (classical)
modular curves.

Using more general Shimura curves, as those from \cite{STV},
one could get similar results for $q=p^{2m}$
with arbitrary $m$.
This case is also mentionned in \cite[Th.~3.1]{Ballet2008}, however,
we point out another error in the proof given there:
in the second half of this proof,
Ballet applies Bertrand's postulate to the primes that correspond
to the levels of these Shimura curves;
but he forgets that, from the very construction of \cite{STV},
which he recalls in his \cite[Prop.~3.1(2)]{Ballet2008},
he has to deal not with the set of all prime numbers,
but only with those that split completely
in a certain abelian extension $L$ of $\Q$.
This splitting condition translates into a certain congruence condition.
Thus, Bertrand's postulate does not apply there.
Still, in principle this strategy of proof could work,
but for this, instead of Bertrand's postulate,
one should substitute an estimate,
such as the one from \cite{Kadiri}, on the gaps between
primes that live in some given residue class.

However, still other families of curves could be used,
for instance Drinfeld modular curves. 
At this stage it is not clear which approach will produce
the best effective bounds.

\vspace{\baselineskip}
\noindent\textbf{3.3. Gaps in the set of values of the Dedekind psi function.}
In Remark~\ref{remgapspsi} we outlined two strategies that could lead
to estimates on $\epsilon_{\psi(\N\setminus p\N)}$, that is,
on gaps in the set of values of the Dedekind psi function (at integers prime
to $p$).
However, quickly we restricted to values of $\psi$ at prime numbers,
so we only had to consider the more studied function $\epsilon_{\mathcal{P}}$.

Obviously, considering all values of $\psi$ instead of only its values at primes,
can only lead to smaller gaps, hence to better bounds on the complexity of multiplication.
The question is: how much better can we get?

Initially, the author hoped to get significantly stronger bounds in this way,
and this hope was one of the reasons for delaying the publication of this work
(see also \cite[Rem.~5.8]{HR-JComp}).
A motivation for this was Corollary~\ref{liste}(ii),
obtained very easily by considering the values of $\psi$ at powers of $2$:
by comparison, it could also have been derived using values at prime numbers,
but this requires Bertrand's postulate, whose proof, given by Chebyshev,
is certainly not so trivial.

Unfortunately, the author is now much more pessimistic, for the following reason.

Very likely, a method that bounds gaps in the set of values of $\psi$,
should apply to a larger class of arithmetic functions.
To any map
\beq
a:\mathcal{P}\longto\Z
\eeq
that takes only finitely many different values,
associate an arithmetic function $f_a$ by the formula
\beq
f_a(N)=N\prod_{\substack{l|N\\ \textrm{$l$ prime}}}\left(1+\frac{a(l)}{l}\right)
\eeq
and let
\beq
\mathcal{S}_a=f_a(\N_{>0})
\eeq
be the set of values of $f_a$.
For instance:
\begin{itemize}
\item if $a(l)=1$ for all $l$, then
\beq
\mathcal{S}_a=\psi(\N_{>0})
\eeq
is the set of all values of the Dedekind psi function
\item if $a(p)=-p$ and $a(l)=1$ for all $l\neq p$, then
\beq
\mathcal{S}_a=\{0\}\cup\psi(\N\setminus p\N)
\eeq
is precisely the set appearing in our application to bilinear complexity
\item if $a(l)=-1$ for all $l$, then
\beq
\mathcal{S}_a=\phi(\N_{>0})
\eeq
is the set of all values of the Euler totient function $\phi$.
\end{itemize}
So we're interested in estimates on the gaps between elements of such a set $\mathcal{S}_a$,
and more precisely, on upper bounds on the associated function $\epsilon_{\mathcal{S}_a}$.
Specializing to values of $f_a$ at primes readily gives
\beq
\epsilon_{\mathcal{S}_a}(x)\leq\epsilon_{\mathcal{P}}(x)+O(1/x)
\eeq
but our hope would be to get a bound significantly sharper.

Now several authors already studied the distribution of the values of $\phi$,
and in particular, in \cite[p.~70]{Ford} it is asked:
``Can it be shown, for example, that for $x$ sufficiently large,
there is a totient between $x$ and $x+x^{1/2}$?''

That means that the inequality $\epsilon_{\phi(\N_{>0})}(x)\leq x^{-1/2}$
is still an open question.
Or said otherwise, one does not know significantly better estimates
on the gaps in the set of all values of $\phi$,
than what one could derive from its values at primes
(compare: $\epsilon_{\mathcal{P}}(x)\leq x^{-0.475}$ by \cite{BHP},
and $\epsilon_{\mathcal{P}}(x)=O(x^{-1/2}\log^2x)$ under RH).

Thus, contrary to the author's initial expectations,
this now leaves very little hope for the similar question for $\psi$.

\vspace{\baselineskip}
\noindent\textbf{3.4. Generalizations of the basic Chudnovsky-Chudnovsky method.}
At the very end of \cite[section~5]{2D-G}, it is discussed how our
optimal solution to Riemann-Roch systems could be combined with
extensions of the Chudnovsky-Chudnovsky method such as \cite{CO},
that use evaluation at points of higher degree and with multiplicities.
This discussion was not reproduced here, because these results are
now superseded by \cite[Th.~5.2(c)]{HR-JComp}, which uses an even finer
notion of generalized evaluation.
Namely, the bounds from \cite{HR-JComp} involve the quantities
\beq
\mu^{\operatorname{sym}}_q(d,u)=\mu^{\operatorname{sym}}_{\F_q}(\F_{q^d}[t]/(t^u))
\eeq
which allow to take into account both higher degree $d$ and multiplicity $u$
at the same time and in the most accurate way.

Still there is a difficulty.
Very often, generalized evaluation,
and more precisely evaluation at points of higher degree,
is used when dealing with curves
that do not have that many points of degree~$1$.
Thus, it becomes useful, for instance,
if one works over a field $\F_p$ of prime order.
However, our method, in Proposition~\ref{constrD-Qet2D-G}
as well as in \cite[Th.~5.2(c)]{HR-JComp}, still requires curves
with sufficiently many points of degree~$1$,
as asked by condition~\eqref{5g}.
In practice, this makes our construction unsuitable for these specific applications,
and instead, one has to revert to suboptimal methods.
A possible solution would be to adapt Proposition~\ref{constrD-Qet2D-G}
and try to make this optimal construction work, say,
with curves having sufficiently many
points of degree~$2$ (instead of degree~$1$);
but this is still an open question.

However, there are two alternative directions where optimality can be reached.

\textbf{3.4.1. Classical bilinear complexity.}
A first direction
is if one is interested in the classical bilinear complexity $\mu_q(k)$,
instead of the symmetric bilinear complexity $\mu^{\operatorname{sym}}_q(k)$.
Note that the classical works \cite{ChCh87}\cite{ChCh88} and \cite{STV}
all dealt only with $\mu_q(k)$, as did also \cite{Ballet2008} and \cite{2D-G}.
Indeed, the symmetric complexity $\mu^{\operatorname{sym}}_q(k)$ was first
introduced in this context only in \cite{HR-JComp}, together with the
importance of the distinction between these two notions.
In particular it is observed there that classical bilinear complexity allows
asymmetric evaluation-interpolation algorithms, whose associated
Riemann-Roch systems are easier to solve optimally.
In this setting, instead of \cite[Th.~5.2(c)]{HR-JComp},
we can use \cite[Th.~5.2(a)]{HR-JComp}, which does not require a condition
like~\eqref{5g} on the number of points of degree~$1$.

For instance, it specializes to the following:
\begin{lemma}
\label{muleqn1+3n2}
Let $X$ be a curve of genus $g$ over a finite field $\F_q$.
Suppose $q\geq7$ and $X$ admits 
\begin{itemize}
\item a closed point $Q$ of degree~$k$
\item $n_1$ closed points of degree~$1$
\item $n_2$ closed points of degree~$2$
\end{itemize}
with
\beq
n_1+2n_2\geq 2k+g-1.
\eeq
Then we have
\beq
\mu_q(k)\leq n_1+3n_2.
\eeq
\end{lemma}
\begin{proof}
Special case
of \cite[Th.~5.2(a)]{HR-JComp} applied with $m=k$, $l=1$, $n_{1,1}=n_1$, $n_{2,1}=n_2$,
and $n_{d,u}=0$ for other values of $d,u$.
\end{proof}
This Lemma~\ref{muleqn1+3n2} repairs Ballet's \cite[Th.~2.1(2)]{Ballet2008},
in the same way Lemma~\ref{mu2k+g-1} repaired Ballet's \cite[Th.~2.1(1)]{Ballet2008}.

We can then continue exactly as in Section~2, with the same modular curves $X_0(N)$,
which we can now consider over the prime field $\F_p$.
Lemma~\ref{courbesmodulaires} gives $g_0(N)\leq\frac{\psi(N)}{12}$
and $2n_2\geq(p-1)\frac{\psi(N)}{12}$,
and with the very same computations we conclude:
\begin{proposition}
Let $p\geq 7$ be a prime number.
Then for all $k>\frac{p+1}{2}$, we have
\beq
\frac{1}{k}\mu_p(k)\leq 3\left(1+\frac{1+\epsilon_{\mathcal{P}}\!\left(\frac{24k}{p-2}\right)}{p-2}\right).
\eeq
\end{proposition}
Again this can be combined with any known upper bound on $\epsilon_{\mathcal{P}}$.
For instance, from \cite{Dudek}
we deduce
\beq
\frac{1}{k}\mu_p(k)\leq 3\left(1+\frac{1+\frac{3}{\left(\frac{24k}{p-2}\right)^{1/3}}}{p-2}\right)
\eeq
for $k\geq\frac{p-2}{24}e^{e^{33.3}}$,
and from \cite{BHP}
we deduce
\beq
\frac{1}{k}\mu_p(k)\leq 3\left(1+\frac{1+\frac{1}{\left(\frac{24k}{p-2}\right)^{0.475}}}{p-2}\right)
\eeq
for $k$ large enough.
Observe also the following asymptotic consequence:
\begin{corollary}
\label{asympmup}
For $p\geq7$ we have
\beq
\limsup_{k\to\infty}\frac{1}{k}\mu_p(k)\leq 3\left(1+\frac{1}{p-2}\right).
\eeq
\end{corollary}

\textbf{3.4.2. Short multiplication of polynomials.}
In a second direction, we observe that
the obstruction discussed at the beginning of \S3.4 applies to evaluation
at points of higher degree, but not to evaluation
with multiplicities (at points of degree~$1$).
Moreover, a new feature introduced in \cite{HR-JComp} is that it
does not only gives a bound \emph{in terms} of the $\mu^{\operatorname{sym}}_q(d,u)$,
it also gives a bound \emph{on} them.
In particular, set
\beq
\widehat{M}^{\operatorname{sym}}_q(l)=\mu^{\operatorname{sym}}_q(1,l)=\mu^{\operatorname{sym}}_{\F_q}(\F_{q}[t]/(t^l)).
\eeq
Multiplication in the quotient algebra $\F_{q}[t]/(t^l)$ is sometimes called
\emph{short multiplication} of polynomials.
Then:
\begin{lemma}
\label{M2l+g-1}
Let $X$ be a curve of genus $g$ over a finite field $\F_q$
with
\beq
|X(\F_q)|>5g.
\eeq
Then for all integers
\beq
\label{encadrel}
l\;\leq\:\frac{|X(\F_q)|+1-g}{2}
\eeq
we have
\beq
\widehat{M}^{\operatorname{sym}}_q(l)\leq 2l+g-1.
\eeq
\end{lemma}
\begin{proof}
Special case
of \cite[Th.~5.2(c)]{HR-JComp} applied with $m=1$, $l=l$, $n_{1,1}=2l+g-1$,
and $n_{d,u}=0$ for other values of $d,u$.
\end{proof}
Lemma~\ref{M2l+g-1} is the exact analogue of Lemma~\ref{mu2k+g-1}
for $\widehat{M}^{\operatorname{sym}}_q(l)$ instead of $\mu^{\operatorname{sym}}_q(k)$.
Mutatis mutandis, we deduce
\beq
\widehat{M}^{\operatorname{sym}}_q(l)\leq2l-1\qquad\textrm{for $l\leq\frac{q}{2}+1$,}
\eeq
\beq
\label{inegShokro}
\widehat{M}^{\operatorname{sym}}_q(l)\leq 2l\qquad\textrm{for $l<\frac{q+e(q)+1}{2}$}
\eeq
and $\widehat{M}^{\operatorname{sym}}_{p^2}(l)$ satisfy the same
upper bounds as $\mu^{\operatorname{sym}}_{p^2}(k)$
in Proposition~\ref{Ballet+} and Corollary~\ref{liste}(i)-(vii).
In particular:
\begin{proposition}
Let $p\geq7$ be prime. Then for all $l$ we have
\beq
\begin{split}
\frac{1}{l}\widehat{M}^{\operatorname{sym}}_{p^2}(l)&\leq 2+\frac{\frac{1}{12}\left\lceil\frac{24l-12}{p-2}\right\rceil_{\psi(\N\setminus p\N)}-1}{l}\\
&\leq 2\left(1+\frac{1+\epsilon_{\mathcal{P}}\!\left(\frac{24l}{p-2}\right)}{p-2}\right).
\end{split}
\eeq
\end{proposition}
Again this can be combined with all existing and future bounds on $\epsilon_{\mathcal{P}}$,
leading for instance to
\beq
\frac{1}{l}\widehat{M}^{\operatorname{sym}}_{p^2}(l)\leq 2\left(1+\frac{1+\frac{3}{\left(\frac{24l}{p-2}\right)^{1/3}}}{p-2}\right)
\eeq
for $l\geq\frac{p-2}{24}e^{e^{33.3}}$, or to
\beq
\frac{1}{l}\widehat{M}^{\operatorname{sym}}_{p^2}(l)\leq 2\left(1+\frac{1+\frac{1}{\left(\frac{24l}{p-2}\right)^{0.475}}}{p-2}\right)
\eeq
for $l$ large enough.

Asymptotically we also deduce the following, which was already observed
(at least implicitely) in \cite[Rem.~6.7]{HR-JComp}:
\begin{corollary}
For $p\geq7$ prime, we have
\beq
\limsup_{l\to\infty}\frac{1}{l}\widehat{M}^{\operatorname{sym}}_{p^2}(l)\leq 2\left(1+\frac{1}{p-2}\right).
\eeq
\end{corollary}

Moreover, as in \S3.4.1, we can also get results over the prime field $\F_p$,
provided we're interested in classical bilinear complexity instead
of symmetric bilinear complexity.
Setting $\widehat{M}_q(l)=\mu_q(1,l)$, the very same approach gives:
\begin{proposition}
Let $p\geq7$ be prime. Then for all $l$ we have
\beq
\frac{1}{l}\widehat{M}_p(l)\leq 3\left(1+\frac{1+\epsilon_{\mathcal{P}}\!\left(\frac{24l}{p-2}\right)}{p-2}\right).
\eeq
\end{proposition}
We leave it to the reader to derive as before the combination with
any bound of his choice on $\epsilon_{\mathcal{P}}$.
\begin{corollary}
For $p\geq7$ prime, we have
\beq
\limsup_{l\to\infty}\frac{1}{l}\widehat{M}_p(l)\leq 3\left(1+\frac{1}{p-2}\right).
\eeq
\end{corollary}

\vspace{\baselineskip}
\noindent\textbf{3.5. Recent work of Ballet and Zykin.}
Very recently Ballet and Zykin published the work~\cite{BZ}.
Although the initial version of their paper did not make reference to \cite{2D-G}
(the final version of \cite{BZ} now repairs this omission)
the core of their proof is precisely the very same argument
that was first introduced there,
using estimates on gaps between primes such as the one
of Baker-Harman-Pintz \cite{BHP}.

Actually, there are two parts in~\cite{BZ}.
The first part, \cite[Prop.~7]{BZ}, concerns a base field $\F_{p^2}$ of prime square order,
so it can be compared directly with our results.
Some differences are quite inessential:
\begin{itemize}
\item We first consider modular curves of arbitrary level $N$,
and then specialize to $N$ prime.
On the other hand, Ballet and Zykin follow \cite{STV} and consider
only level $11N$ (or $23N$). The curves produced in this way thus form
a slightly less dense family.
\item In passing from Proposition~\ref{Ballet+} to Corollary~\ref{liste}(i),
we kept only the term proportional to $k$ and we discarded the constant
term.
This gives a simpler expression, although slightly less precise.
On the other hand, Ballet and Zykin kept track of this constant term.
\item The strongest bounds in \cite[Cor.~28]{2D-G} and in \cite[Prop.~7]{BZ}
both are based on the estimate of Baker-Harman-Pintz \cite{BHP}.
Weaker but more explicit bounds are also proposed using alternative estimates.
In particular Ballet and Zykin refer to Dudek's estimate \cite{Dudek},
which did not exist at the time when \cite{2D-G} was written,
but is now included for completeness as Corollary~\ref{liste}(vii),
in \S3.1 above.
As explained there, any further progress on gaps between primes
automatically translates into a bound on multiplication complexity.
\end{itemize}
All the details are essentially negligible.
However there is another, much more important difference:
\begin{itemize}
\item Beside gaps between primes, a second ingredient in our work is our
optimal solution to Riemann-Roch systems.
Thanks to this, our uniform bounds match the best asymptotic
bound~\eqref{rSTV}.
On the other hand, Ballet and Zykin use a suboptimal construction,
which allow them only to match the weaker asymptotic bound \eqref{asympChCh},
as they explicitly state in \cite[Prop.~7(3)]{BZ}.
\end{itemize}
Because of this, essentially all results in the first part of \cite{BZ} are already included in our stronger Corollary~\ref{liste}.
More precisely, only one very specific case of \cite[Prop.~7]{BZ} is not covered, namely the case $q=25$.

\vspace{\baselineskip}

On the other hand, the second part of \cite{BZ} considers a base field of prime order.
As discussed at the beginning of \S3.4, 
our optimal method for solving Riemann-Roch systems does not work
well for symmetric algorithms over prime fields.
Instead, to prove \cite[Prop.~10]{BZ}
Ballet and Zykin use a suboptimal method from~\cite{BR2004},
directly adapted from the original method of \cite{ChCh87}\cite{ChCh88}.
This is probably the best that could be done with the current state
of knowledge, and \cite[Prop.~10]{BZ} is not covered by the present work.

Now it is interesting to compare the asymptotic bound they get this way for symmetric complexity \cite[Prop.~10(3)]{BZ}
\beq
\limsup_{k\to\infty}\frac{1}{k}\mu^{\operatorname{sym}}_p(k)\leq 3\left(1+\frac{4/3}{p-3}\right)
\eeq
with our Corollary~\ref{asympmup} that holds for classical bilinear complexity.
This suggests that,
if one could solve the problem alluded to at the beginning of \S3.4,
this would lead to uniform bounds on the symmetric complexity matching
the much better, but still conjectural, asymptotic bound
\beq
\limsup_{k\to\infty}\frac{1}{k}\mu^{\operatorname{sym}}_p(k)\leq 3\left(1+\frac{1}{p-2}\right).
\eeq

\end{document}